\documentclass[a4paper,reqno]{amsart}

\usepackage[foot]{amsaddr}
\usepackage{graphicx,enumerate,nicefrac,color,bm,cite}
\usepackage{amsmath, amssymb, amstext, amsfonts}
\usepackage{subfig}
\usepackage{pgfplots}

\usepackage{a4wide}

\usepackage{algorithm,algpseudocode,tabularx}

\makeatletter
\newcommand{\multiline}[1]{%
  \begin{tabularx}{\dimexpr\linewidth-\ALG@thistlm}[t]{@{}X@{}}
    #1
  \end{tabularx}
}
\makeatother

\newcommand{\norm}[1]{\left\|#1\right\|}
\newcommand{\ltn}[1]{\norm{#1}_{L^2(\Omega)}}
\newcommand{\operator}[1]{\mathsf{#1}}
\newcommand{\A}{\operator{A}}
\newcommand{\E}{\operator{E}}

\newcommand{\F}{\operator{F}}

\newcommand{\uex}{u_{X,\epsilon}^\star}

\newcommand{\dd}{\mathsf{d}}
\newcommand{\dt}{\,\mathsf{d} t}

\newcommand{\omegamin}{\Omega_{\epsilon,-}}
\newcommand{\omegaplus}{\Omega_{\epsilon,+}}

\newcommand{\x}{\bm{\mathsf{x}}}
\newcommand{\epsm}{\epsilon_{-}}
\newcommand{\epsp}{\epsilon_{+}}
\newcommand{\xim}{\xi'_{-}}

\newcommand{\xiem}{\xi'_{\epsilon,-}}
\newcommand{\xiep}{\xi'_{\epsilon,+}}
\newcommand{\un}{u_N^\star}
\newcommand{\re}{r}
\newcommand{\se}{s}
\newcommand{\une}{u_{N,\epsilon}^\star}

\newcommand{\hs}{H_0^1(\Omega)}
\newcommand{\wop}{W_0^{1,p(\cdot)}(\Omega)}
\newcommand{\lps}{L^{p(\cdot)}(\Omega)}

\newcommand{\muem}{\mu_{\epsilon,-}}
\newcommand{\muep}{\mu_{\epsilon,+}}

\newcommand{\dx}{\,\mathsf{d} x}
\newcommand{\bkappa}{\boldsymbol{\kappa}}
\newcommand{\btau}{\boldsymbol{\tau}}
\newcommand{\dprod}[1]{\left<#1\right>}

\DeclareMathOperator{\Div}{div}

\DeclareMathOperator*{\esssup}{ess\,sup}

\newcommand{\nnnn}[1]{\norm{#1}_{L^2(\Omega)}}
\newcommand{\wps}{W_0^{1,p(\cdot)}(\Omega)}

\DeclareMathOperator*{\argmin}{arg\,min}

\newtheorem{theorem}{Theorem}[section]
\newtheorem{lemma}[theorem]{Lemma}
\newtheorem{proposition}[theorem]{Proposition}

\theoremstyle{definition}
\newtheorem{remark}[theorem]{Remark}
\newtheorem{experiment}{Experiment}[section]

\title[A relaxed, damped Ka\v{c}anov scheme for the $p(x)$-Poisson equation]{A damped Ka\v{c}anov scheme for the numerical solution of a relaxed $p(x)$-Poisson equation}

\author{Pascal Heid}
\email{pascal.heid@ma.tum.de}
\address{Department of Mathematics, Technical University of Munich, Boltzmannstr.~3, 85748 Garching bei München, Germany \newline Munich Center for Machine Learning (MCML)}



\begin{document}

\begin{abstract}
The focus of the present work is the (theoretical) approximation of a solution of the $p(x)$-Poisson equation. To devise an iterative solver with guaranteed convergence, we will consider a relaxation of the original problem in terms of a truncation of the nonlinearity from below and from above by using a pair of positive cut-off parameters. We will then verify that, for any such pair, a damped Ka\v{c}anov scheme generates a sequence converging to a solution of the relaxed equation. Subsequently, it will be shown that the solutions of the relaxed problems converge to the solution of the original problem in the discrete setting. Finally, the discrete solutions of the unrelaxed problem converge to the continuous solution. Our work will finally be rounded up with some numerical experiments that underline the analytical findings. 
\end{abstract}

\maketitle

\section{Introduction}

In this work, we consider the $p(x)$-Laplace problem
\begin{equation}\label{eq:pxpoisson}
\begin{aligned}
-\Div(|\nabla u|^{p(x)-2}\nabla u)&=f(x,u) \quad && x\in \Omega,\\
u(x)&=0 \quad && x\in \partial \Omega,
\end{aligned}
\end{equation}
where $\Omega \subset \mathbb{R}^d$, $d \in \{2,3\}$, is an open and bounded domain with Lipschitz boundary. We emphasise that problems of the form~\eqref{eq:pxpoisson} have been widely applied, for instance in image processing~\cite{CLS:2004,CLR:2006,Tiirola:2014,YiGe:2017,KARAMI2018534,ZSGW:2019,DKKM:2022}, electrorheological fluids~\cite{RAJAGOPAL1996401,MR1810360,diening2002theoretical}, or magnetostatics~\cite{CEKIC2012838}. The weak formulation of~\eqref{eq:pxpoisson} is given as follows: Find $u \in W_0^{1,p(\cdot)}(\Omega)$ such that 
\begin{equation}\label{eq:weakpx}
\int_\Omega |\nabla u|^{p(x)-2} \nabla u \cdot \nabla v \dx=\int f(x,u)v \dx \qquad \text{for all} \ v \in W_0^{1,p(\cdot)}(\Omega),
\end{equation}
where, in this work,
\begin{align}\label{eq:cplus}
p \in C_{+}(\overline{\Omega}):=\left\{f \in C(\overline{\Omega}): \inf_{x \in \overline{\Omega}}f(x) > 1\right\};
\end{align}
we note that $p_{-} = \inf_{x \in \overline{\Omega}} p(x)>1$ and $p_{+}=\sup_{x \in \overline{\Omega}} p(x)<\infty$. 
It is known that problem~\eqref{eq:weakpx} may have multiple, even infinitely many, solutions. For the context of our analysis, we want to restrict to the case where the solution is unique. Indeed,  a solution of~\eqref{eq:weakpx} exists and is unique if the source function $f$ in~\eqref{eq:weakpx} satisfies the assumption ($A_\alpha$) below; we refer to~\cite[Thm.~4.2]{FanZhang:03} for a proof of the statement.
\begin{itemize}
\item[($A_\alpha$)] It holds that $f(x,u)=f(x) \in L^{\alpha(\cdot)}(\Omega)$ with $\alpha \in C_{+}(\overline{\Omega})$, cf.~\eqref{eq:cplus}, and $\nicefrac{1}{\alpha(x)}+\nicefrac{1}{p^\star(x)}<1$ for all $x \in \Omega$; here,
\begin{align} \label{eq:pstar}
p^\star(x)=\begin{cases} 
\nicefrac{d p(x)}{d-p(x)} & p(x) < d,\\
\infty & p(x) \geq d.
\end{cases}
\end{align}
\end{itemize}
Furthermore, upon defining the (energy) functional $\E:\wps \to \mathbb{R}$ by
\begin{align} \label{eq:penergy}
\E(u)=\int_\Omega \frac{1}{p(x)} |\nabla u|^{p(x)} \dx-\int_\Omega f u \dx, \qquad u \in \wps,
\end{align}
we have that the unique solution $u^\star \in \wps$ of~\eqref{eq:weakpx} is, equivalently, the unique minimiser of $\E$; i.e.,
\begin{align} \label{eq:pxopt}
u^\star=\argmin_{u \in \wps}\E(u).
\end{align}
In particular, our original (weak) equation~\eqref{eq:weakpx} is the Euler--Lagrange equation of the optimisation problem~\eqref{eq:pxopt}. Solving for $u^\star$, either by minimising the energy functional~\eqref{eq:penergy} or by considering the $p(x)$-Poisson problem~\eqref{eq:weakpx}, is a highly challenging problem. For $1<p(x) \equiv p<2$ and $2 \leq p(x) \equiv p$, respectively, a relaxed problem was introduced in~\cite{Diening:2020} and~\cite{BalciDieningStorn:2022}, respectively, which can be iteratively solved by the Ka\v{c}anov scheme. In those references it is further shown that the unique solution of the relaxed problem converges, in a certain sense, to the solution of the original problem. However, the problem considered in this work is tremendously more challenging, since we allow for a variable exponent $p(\cdot)$, which may take values below and above the threshold value $p=2$. For that reason, we will settle for some weaker convergence results, as outlined below. In particular, we will prove that a damped Ka\v{c}anov iteration scheme converges to the unique solution of the relaxed problem. Subsequently, we will show in the discrete setting that the unique solutions of the relaxed problem converges to the unique solution of unrelaxed $p(x)$-Poisson problem. Finally, we will verify the convergence of the discrete (unrelaxed) solution to the unique solution of the continuous problem.

\subsubsection*{Outline} In Section 2 we present the necessary notation, recall some well-known results, introduce the relaxed problem and provide some preliminary results, which will be crucial in the analysis in the later sections. The third section deals with the damped Ka\v{c}anov iteration scheme for the solution of the relaxed problem. Subsequently, in Section 4, the convergence of the solution of the relaxed problem to the one of the unrelaxed problem is verified in the discrete setting. In addition, we prove the convergence of the discrete solution to the continuous solution. Some numerical experiments are then performed in Section 5, before our work is concluded with some final remarks in Section 6.

\section{Preliminaries}

Throughout our work, we will assume that $\Omega \subset \mathbb{R}^d$, $d \in \{2,3\}$, is an open and bounded domain with Lipschitz boundary.

\subsection{Basic notions}

In the given work we will consider Lebesgue and Sobolev spaces (with variable exponents). As usual, for any $p \in [1,\infty)$ we denote by $L^p(\Omega)$ the Lebesgue space of $p$-integrable functions with corresponding norm $\norm{f}_{L^p(\Omega)}:=\left(\int_\Omega |f(x)|^p \dx\right)^{\nicefrac{1}{p}}$.  Furthermore, $L^\infty(\Omega)$ signifies the Lebesgue space of essentially bounded, measurable functions on $\Omega$ endowed with the norm $\norm{f}_{L^\infty(\Omega)}:=\esssup_{x \in \Omega} |f(x)|$. Likewise, for $p \in [1,\infty]$, we denote by $W_0^{1,p}(\Omega)$ the space of Sobolev functions with zero trace along the boundary $\partial \Omega$, endowed with the norm $\norm{f}_{W_0^{1,p}(\Omega)}:=\norm{\nabla f}_{L^p(\Omega)}$. As usual, for $p=2$, we use the convention $H^1_0(\Omega):=W_0^{1,2}(\Omega)$.

Next, we will introduce the Lebesgue and Sobolev spaces with variable exponents; for a very extensive treatment of those spaces we refer the interested reader to the monograph~\cite{DieningEtAl:2011}. For any given measurable function $p:\Omega \to [1,\infty]$ we introduce the Lebesgue space with variable exponent
\begin{align*} 
L^{p(\cdot)}(\Omega):=\left\{f:\Omega \to \mathbb{R} \ \text{measurable}:\int_\Omega |f(x)|^{p(x)} \dx < \infty\right\},
\end{align*}
endowed with the Luxemburg norm
\begin{align} \label{eq:pxnorm}
\norm{f}_{L^{p(\cdot)}(\Omega)}:=\inf \left\{\lambda >0: \int_\Omega \left|\frac{f(x)}{\lambda}\right|^{p(x)} \dx \leq 1\right\}.
\end{align}
We emphasise that $(L^{p(\cdot)}(\Omega),\norm{\cdot}_{L^{p(\cdot)}(\Omega)})$ is a separable and reflexive Banach space, see, e.g., \cite{FAN2001424}. Furthermore, for a constant exponent $p$ the definition of the Luxemburg norm in~\eqref{eq:pxnorm} coincides with the usual $L^p(\Omega)$-norm. 
In a similar manner we define the Sobolev space with variable exponent
\begin{align*} 
W^{1,p(\cdot)}(\Omega):=\left \{f \in L^{p(\cdot)}(\Omega): \nabla f \in L^{p(\cdot)}(\Omega) \right\},
\end{align*}
equipped with the norm
\begin{align*}
\norm{f}_{W^{1,p(\cdot)}(\Omega)}:=\norm{\nabla f}_{\lps}+\norm{f}_{\lps}.
\end{align*}
Moreover, for $p \in C_{+}(\overline{\Omega})$, we further consider the Sobolev space with variable exponent and zero boundary values $\wop$, which is the closure of $C_0^\infty(\Omega)$ in $W^{1,p(\cdot)}(\Omega)$; this spaces will be endowed with the norm
\begin{align*}
\norm{f}_{\wop}:=\norm{\nabla f}_{\lps}.
\end{align*}
Those are again separable and reflexive Banach spaces, see~\cite[Sec.~8]{DieningEtAl:2011}. 

\subsection{Auxiliary results}
We shall now state some preliminary results concerning Lebesgue and Sobolev spaces with variable exponents that are well-known in the literature.

\begin{proposition} \label{eq:embedding}
Let $r,s \in C_{+}(\overline{\Omega})$. If $r \leq s$ almost everywhere, then we have the continuous embeddings $L^{s(\cdot)}(\Omega) \hookrightarrow L^{r(\cdot)}(\Omega)$ and $W^{1,s(\cdot)}_0(\Omega) \hookrightarrow W^{1,r(\cdot)}_0(\Omega)$. 
\end{proposition} 

We refer, for instance, to~\cite[Thm.~1.11]{FAN2001424} or~\cite[Sec.~3.3]{DieningEtAl:2011}. Moreover, we also have a H\"{o}lder inequality for Lebesgue spaces with variable exponents; we refer to~\cite[Lem.~3.2.20]{DieningEtAl:2011}.

\begin{proposition} \label{eq:holder}
Let $q,r,s:\Omega \to [1,\infty]$ measurable with 
\begin{align*}
\frac{1}{q(x)}+\frac{1}{r(x)}=\frac{1}{s(x)} \qquad \text{for almost every} \ x \in \Omega.
\end{align*}
Then, for all $f \in L^{q(\cdot)}(\Omega)$ and $g \in L^{r(\cdot)}(\Omega)$ we have that $fg \in L^{s(\cdot)}(\Omega)$ with
\begin{align*}
\norm{fg}_{L^{s(\cdot)}(\Omega)} \leq 2 \norm{f}_{L^{q(\cdot)}(\Omega)} \norm{g}_{L^{r(\cdot)}(\Omega)}.
\end{align*}
Especially, for measurable $p:\Omega \to [1,\infty]$, denote by $p'$ its H\"{o}lder conjugate; i.e.
\begin{align*} 
\frac{1}{p(x)}+\frac{1}{p'(x)}=1 \qquad \text{for almost every} \ x \in \Omega.
\end{align*}
Then, for any $f \in L^{p(\cdot)}(\Omega)$ and $g \in L^{p'(\cdot)}(\Omega)$, we have that $fg \in L^1(\Omega)$ with
\begin{align*} 
\norm{fg}_{L^1(\Omega)} \leq 2 \norm{f}_{L^{p(\cdot)}(\Omega)}\norm{g}_{L^{p'(\cdot)}(\Omega)}.
\end{align*}
\end{proposition}

The next result states a property that is equivalent to the convergence in $L^{p(\cdot)}(\Omega)$, see~\cite[Thm~1.4]{FAN2001424}.

\begin{proposition} \label{prop:equiv}
Let $f_k,f \in L^{p(\cdot)}(\Omega)$, $k \in \mathbb{N}$, and $p \in C_{+}(\overline{\Omega})$, cf.~\eqref{eq:cplus}. Then, the following statements are equivalent:
\begin{enumerate}[(a)]
\item $\lim_{k \to \infty} \norm{u_k-u}_{L^{p(\cdot)}(\Omega)}=0$;
\item $u_k$ converges in measure to $u$ and $\lim_{k \to \infty} \norm{u_k}_{L^{p(\cdot)}(\Omega)}=\norm{u}_{L^{p(\cdot)}(\Omega)}$.
\end{enumerate}
\end{proposition}

%

Finally, we also have a Sobolev embedding for variable exponent spaces, see, e.g.,~\cite[Cor.~8.3.2]{DieningEtAl:2011}.

\begin{proposition} \label{prop:embedding}
Let $p \in C_{+}(\overline{\Omega})$ and $q:\Omega \to [1,\infty]$ measurable such that $q \leq p^\star$ a.e., where $p^\star$ is defined as in~\eqref{eq:pstar}. Then, the embedding $W^{1,p(\cdot)}(\Omega) \hookrightarrow L^{q(\cdot)}(\Omega)$ is continuous, with the embedding constant only depending on $p,q$ and $\Omega$. 
\end{proposition}

\subsection{Relaxed $p(x)$-Poisson problem}
From now on, we will always assume that $p \in C_{+}(\overline{\Omega})$. As mentioned in the introduction, we shall consider a relaxed version of the problem~\eqref{eq:weakpx} as introduced in~\cite{Diening:2020, BalciDieningStorn:2022}; in those references, however, the exponent was constant. The relaxation is based on two cut-off parameters $0 < \epsm < 1 <\epsp < \infty$; for simplicity, we signify the pair of cut-off parameters $\epsm,\epsp$ by $\epsilon$. In the following, let $X \subseteq \hs$ be a closed subspace; we are especially interested in the cases $X=\hs$ or $X$ being finite dimensional. Then, the relaxed problem is stated as follows:  Find $u \in X$ such that 
\begin{align} \label{eq:weakrelaxed}
\int_\Omega (\epsilon_{-} \lor |\nabla u| \land \epsp)^{p(x)-2} \nabla u \cdot \nabla v \dx = \int_\Omega fv \dx \qquad \text{for all} \ v \in X.
\end{align}
Upon defining the function $\mu_\epsilon:\Omega \times \mathbb{R}_{\geq 0} \to \mathbb{R}_{\geq 0}$ by 
\begin{align} \label{eq:relaxedmu}
\mu_\epsilon(x,t):=\begin{cases}
\epsm^{p(x)-2} & t<\epsm^2,\\
t^{\frac{p(x)-2}{2}} & \epsm^2\leq t \leq \epsp^2, \\
\epsp^{p(x)-2} & t > \epsp^2,
\end{cases}
\end{align}
the relaxed problem~\eqref{eq:weakrelaxed} can be stated equivalently as
\begin{align} \label{eq:rwmuproblem}
\text{Find} \ u \in X: \qquad \int_\Omega \mu_\epsilon(x,|\nabla u|^2) \nabla u \cdot \nabla v \dx=\int_\Omega fv \dx \qquad \text{for all} \ v \in X.
\end{align}
Here, to guarantee the well-posedness of the right-hand side, we need to impose some (possibly) stronger assumptions on the source function $f$. Indeed, thanks to the Sobolev embedding theorem and H\"{o}lders inequality, the integral on the right-hand side of~\eqref{eq:rwmuproblem} is well-defined if 
\begin{align} \label{eq:alpahmin}
\begin{cases}
\alpha_{-}:=\min_{x \in \overline{\Omega}}\alpha(x)>1 & \text{for} \ d=2, \\
\alpha_{-} \geq \frac{6}{5} &\text{for} \ d=3,
\end{cases}
\end{align}
where, as before, $f \in L^{\alpha(\cdot)}(\Omega) \subseteq L^{\alpha_{-}}(\Omega)$, cf.~Proposition~\ref{eq:embedding}, with $\alpha \in C_{+}(\overline{\Omega})$.

As in the unrelaxed case, the relaxed problem~\eqref{eq:rwmuproblem} arises as the Euler--Lagrange equation of an (energy) minimisation problem
\begin{align*} 
\argmin_{u \in X} \E_\epsilon(u),
\end{align*}
where 
\begin{align} \label{eq:relaxedE}
\E_\epsilon(u):=\int_\Omega \varphi_\epsilon(x,|\nabla u|^2) \dx - \int_\Omega f u \dx;
\end{align}
here, for all $x \in \Omega$,
\begin{align*} 
\varphi_\epsilon(x,t):=\begin{cases}
\frac{1}{2}\epsm^{p(x)-2}t+\left(\frac{1}{p(x)}-\frac{1}{2}\right)\epsm^{p(x)} & 0 \leq t< \epsm^2,\\
\frac{1}{p(x)} t^{\nicefrac{p(x)}{2}} & \epsm^2 \leq t \leq \epsp^2, \\
\frac{1}{2}\epsp^{p(x)-2}t+\left(\frac{1}{p(x)}-\frac{1}{2}\right)\epsp^{p(x)} & \epsp^2<t.\\
\end{cases}
\end{align*}

Next, we want to show that~\eqref{eq:rwmuproblem} has a unique solution, which is equivalently the unique minimiser of~\eqref{eq:relaxedE}. For that purpose, and to formulate the damped Ka\v{c}anov scheme in Section~\ref{sec:kacanov}, we shall introduce the operators $\A_\epsilon[\cdot](\cdot):X \times X \to X'$ and $\ell_f \in X'$, where $X'$ signifies the dual space of $X$, which are defined by
\begin{align}\label{eq:Aoperator}
\A_\epsilon[u](v)(w):=\int_\Omega \mu_\epsilon(x,|\nabla u|^2) \nabla v \cdot \nabla w \dx, \qquad u,v,w \in X,
\end{align}
and
\begin{align*}
\ell_{f}(v):=\int_\Omega fv \dx, \qquad v \in X,
\end{align*} 
respectively. Then, upon defining the (residual) operator $\F_\epsilon:X \to X'$ by
\begin{align}\label{eq:Fop}
\F_\epsilon(u):=\A_\epsilon[u](u)-\ell_{f}, \qquad u \in X,
\end{align}
the relaxed problem~\eqref{eq:rwmuproblem} is equivalent to determining an element $u \in X$ such that
\begin{align} \label{eq:Fweakproblem}
\dprod{\F_\epsilon(u),v}=0 \qquad \text{for all} \ v \in X,
\end{align}
where $\dprod{\cdot,\cdot}$ denotes the duality product of $X$ and its dual space $X'$. We further note that~\eqref{eq:Fweakproblem} is equivalent to $\F_\epsilon(u)=0$ in $X'$. Moreover, we have that $\E_\epsilon'=\F_\epsilon$; i.e., $\E_\epsilon$ is the potential of $\F_\epsilon$. 

In order to show the existence and uniqueness of a solution of~\eqref{eq:rwmuproblem}, we will verify that the operator $\F_\epsilon$ is Lipschitz continuous and strongly monotone. To that end, we will first examine the coefficient $\mu(\cdot,\cdot)$. By definition, cf.~\eqref{eq:relaxedmu}, we have that
\begin{align} \label{eq:mubound}
\mu_{-} \leq \mu(x,t) \leq \mu_{+} \qquad \text{for all} \ (x,t) \in \Omega^\star,
\end{align}
where $\Omega^\star=\Omega \times \mathbb{R}_{\geq 0}$ and
\begin{align*} 
\mu_{\epsilon,-}:=\min\left\{\epsm^{p_{+}-2},\epsp^{p_{-}-2}\right\} \quad \text{and} \quad \mu_{\epsilon,+}:=\max \left\{\epsm^{p_{-}-2},\epsp^{p_{+}-2}\right\}.
\end{align*}
Morever, let $\xi_\epsilon:\Omega \times \mathbb{R}_{\geq 0} \to \mathbb{R}_{\geq 0}$ be defined by
\begin{align} \label{eq:xiepsded}
\xi_\epsilon(x,t):=\mu_\epsilon(x,t^2)t \qquad \text{for all} \ x \in \Omega, \ t \geq 0;
\end{align}
this function will be decisive for our analysis below. We note that, for any given $x \in \Omega$, the function $t \mapsto \xi_\epsilon(x,t)$ is continuous and differentiable for $t \not \in \{\epsm,\epsp\}$. Specifically, we have that
\begin{align*} 
\xi_\epsilon'(x,t)=\begin{cases}
\epsm^{p(x)-2} &  t < \epsm, \\
(p(x)-1)t^{p(x)-2} &  \epsm<t<\epsp, \\
\epsp^{p(x)-2} &  t > \epsp,
\end{cases}
\end{align*}
where $\xi_\epsilon'$ denotes the derivative with respect to the second variable. In particular, for given $x \in \Omega$, the mapping $t \mapsto \xi_\epsilon'(x,t)$ is even continuous for $t \not \in \{\epsm,\epsp\}$. In addition we have that
\begin{align} \label{eq:xipmin}
\xi'_{\epsilon,-}:=\inf_{(x,t) \in \Omega^\star} \xi_\epsilon'(x,t)>0 \quad \text{and} \quad \xi'_{\epsilon,+}:=\sup_{(x,t) \in \Omega^\star} \xi_\epsilon'(x,t)<\infty;
\end{align}
here and in the following, for any $x \in \Omega$, we extend the function $\xi'_{\epsilon}(x,t)$ to $t \in \{\epsm,\epsp\}$ by setting $\xi'_{\epsilon}(x,\epsilon_{\pm}):=(p(x)-1)\epsilon_{\pm}^{p(x)-2}$. Hence, by applying the mean value theorem (in a piecewise manner), we find that
\begin{align} \label{eq:muprop}
\xiem(t-s) \leq \mu(x,t^2)t-\mu(x,s^2)s \leq \xiep (t-s), \qquad t \geq s \geq 0,
\end{align}
which, in turn, implies $\xiem \leq \muem$ and $\xiep \geq \muep$. 

\begin{remark} 
For the sake of completeness, we shall make the bounds in \eqref{eq:xipmin} explicit, which requires to distinguish three cases:\\
If $1 < p_{-} \leq p_{+} < 2$, then: 
\begin{align*}
\inf_{(x,t) \in \Omega^\star} \xi_\epsilon'(x,t) \geq (p_{-}-1) \epsp^{p_{-}-2} \qquad \text{and} \qquad \sup_{(x,t) \in \Omega^\star} \xi_\epsilon'(x,t) \leq \epsm^{p_{-}-2}.
\end{align*}
If $2 < p_{-} \leq p_{+} <\infty$, then: 
\begin{align*}
\inf_{(x,t) \in \Omega^\star} \xi_\epsilon'(x,t) \geq \epsm^{p_{+}-2} \qquad \text{and} \qquad 
\sup_{(x,t) \in \Omega^\star} \xi_\epsilon'(x,t) \leq (p_{+}-1) \epsp^{p_{+}-2}.
\end{align*}
If $1 < p_{-}  \leq 2 \leq p_{+} < \infty$, then: 
{\small \begin{align*}
\inf_{(x,t) \in \Omega^\star} \xi_\epsilon'(x,t) \geq (p_{-}-1) \min \left\{\epsm^{p_{+}-2},\epsp^{p_{-}-2}\right\} \quad \text{and} \quad
\sup_{(x,t) \in \Omega^\star} \xi_\epsilon'(x,t) \leq (p_{+}-1) \max \left\{\epsm^{p_{-}-2},\epsp^{p_{+}-2}\right\}.
\end{align*}}
\end{remark}

Our next result, which is the main step towards to strong monotonicity and the Lipschitz continuity of $\F_\epsilon$, is, together with its proof, largely borrowed from~\cite[Lem.~2.1]{HeidSuli:2022}, which in turn is based on~\cite[Lem.~3.1]{BarrettLiu:1993}.

\begin{lemma} \label{lem:help}
Let $\mu_\epsilon$ and $\xi_\epsilon$ be defined as in~\eqref{eq:relaxedmu} and~\eqref{eq:xiepsded}, respectively. Then, for $\bkappa,\btau \in \mathbb{R}^{d}$, we have for any given $x \in \Omega$ that 
\begin{align} \label{eq:mukappaU}
\left|\mu_\epsilon(x,|\bkappa|^2)\bkappa-\mu_\epsilon(x,|\btau|^2)\btau\right|^2 \leq 3 (\xiep)^2 |\bkappa-\btau|^2
\end{align}
and 
\begin{align} \label{eq:mukappal}
(\mu_\epsilon(x,|\bkappa|^2)\bkappa-\mu_\epsilon(x,|\btau|^2)\btau):(\bkappa-\btau) \geq \xiem  |\bkappa-\btau|^2.
\end{align}
\end{lemma}

\begin{proof}
Let us start with the proof of the bound~\eqref{eq:mukappaU}. First, we note that
\begin{multline*}
\left|\mu_\epsilon(x,|\bkappa|^2)\bkappa-\mu_\epsilon(x,|\btau|^2)\btau\right|^2 \\ 
\begin{aligned}
&=\mu_\epsilon(x,|\bkappa|^2)^2|\bkappa|^2+ \mu_\epsilon(x,|\btau|^2)^2|\btau|^2-2\mu_\epsilon(x,|\bkappa|^2)\mu_\epsilon(x,|\btau|^2)\bkappa \cdot \btau\\
&=\left(\mu_\epsilon(x,|\bkappa|^2)|\bkappa|- \mu_\epsilon(x,|\btau|^2)|\btau| \right)^2+ 2\mu_\epsilon(x,|\bkappa|^2)\mu_\epsilon(x,|\btau|^2)(|\bkappa||\btau|-\bkappa \cdot \btau) \\
&\leq \left(\mu_\epsilon(x,|\bkappa|^2)|\bkappa|- \mu_\epsilon(x,|\btau|^2)|\btau| \right)^2+ 2\mu_\epsilon(x,|\bkappa|^2)\mu_\epsilon(x,|\btau|^2)|\bkappa-\btau|^2.
\end{aligned}
\end{multline*}
In turn, the upper bounds in~\eqref{eq:muprop} and~\eqref{eq:mubound} imply that
\begin{align*}
\left|\mu_\epsilon(x,|\bkappa|^2)\bkappa-\mu_\epsilon(x,|\btau|^2)\btau\right|^2  \leq (\xiep)^2 (|\bkappa|-|\btau|)^2+2\muep^2|\bkappa-\btau|^2 \leq 3 (\xiep)^2 |\bkappa-\btau|^2,
\end{align*}
where we used in the second inequality that $\xiep \geq \muep$. 

Now we will take care of the inequality~\eqref{eq:mukappal}. A simple and straightforward calculation reveals that
\begin{align} 
(\mu_\epsilon(|\bkappa|^2)\bkappa-\mu_\epsilon(|\btau|^2)\btau) \cdot (\bkappa-\btau) &= \mu_\epsilon(|\bkappa|^2)|\bkappa|^2+\mu_\epsilon(|\btau|^2)|\btau|^2-\mu_\epsilon(|\bkappa|^2)\bkappa \cdot \btau-\mu_\epsilon(|\btau|^2)\btau \cdot \bkappa \nonumber \\ 
&=\left(\mu_\epsilon(|\bkappa|^2)|\bkappa|-\mu_\epsilon(|\btau|^2)|\btau|\right)(|\bkappa|-|\btau|) \label{eq:firstsummand} \\
& \quad + (\mu_\epsilon(|\bkappa|^2)+\mu_\epsilon(|\btau|^2))(|\bkappa||\btau|-\bkappa\cdot \btau). \label{eq:secondsummand}
\end{align}
The term in~\eqref{eq:firstsummand} can be bounded from below, thanks to~\eqref{eq:muprop}, by
\begin{align} \label{eq:boundfirstsummand}
\left(\mu_\epsilon(|\bkappa|^2)|\bkappa|-\mu_\epsilon(|\btau|^2)|\btau|\right)(|\bkappa|-|\btau|) \geq \xiem (|\bkappa|-|\btau|)^2
\end{align}
Concerning~\eqref{eq:secondsummand}, we find by considering the bound~\eqref{eq:mubound} that
\begin{align} \label{eq:boundsecondsummand}
(\mu_\epsilon(|\bkappa|^2)+\mu_\epsilon(|\btau|^2))(|\bkappa||\btau|-\bkappa \cdot \btau) \geq 2 \muem (|\bkappa||\btau|-\bkappa \cdot \btau) \geq 2 \xiem (|\bkappa||\btau|-\bkappa \cdot \btau). 
\end{align}
Consequently, using the established bounds~\eqref{eq:boundfirstsummand} and~\eqref{eq:boundsecondsummand} for the summands in~\eqref{eq:firstsummand} and~\eqref{eq:secondsummand}, respectively, yields
\begin{align*}
(\mu_\epsilon(|\bkappa|^2)\bkappa-\mu_\epsilon(|\btau|^2)\btau) \cdot (\bkappa-\btau) &\geq \xim \left( (|\bkappa|-|\btau|)^2+2\left(|\bkappa||\btau|-\bkappa \cdot \btau\right) \right) \\
&=\xiem\left(|\bkappa|^2+|\btau|^2-2 \bkappa \cdot \btau\right) \\
&=\xiem|\bkappa-\btau|^2;
\end{align*}
this finishes the proof.
\end{proof}

Based on this lemma, one can derive the strong monotonicity and the Lipschitz continuity of the mapping $u \mapsto \A_\epsilon[u](u)$, and equivalently of $u \to \F_\epsilon(u)$ since $\F_\epsilon(u)-\F_\epsilon(v)=\A_\epsilon[u](u)-\A_\epsilon[v](v)$.

\begin{proposition} [{\hspace{0.1pt} \cite[Prop.~2.2]{HeidSuli:2022}}] \label{lem:properties}
Let $\mu_\epsilon, \xi_\epsilon, \A_\epsilon$, and $\F_\epsilon$ be defined as in~\eqref{eq:relaxedmu},~\eqref{eq:xiepsded},~\eqref{eq:Aoperator}, and~\eqref{eq:Fop}, respectively.
\begin{enumerate}[(a)]
\item For given $u \in X$, $\A_\epsilon[u](\cdot)(\cdot)$ is a uniformly bounded and coercive, symmetric bilinear form on $X \times X$. In particular, the following inequalities hold:
\begin{align} \label{eq:Abd}
\A_\epsilon[u](v)(w) \leq \muep \ltn{\nabla v}\ltn{\nabla w}
\end{align}
and 
\begin{align} \label{eq:coercive}
\A_\epsilon[u](v)(v) \geq \muem \ltn{ \nabla v}^2
\end{align}   
for any $u,v,w \in X$.
\item The mappings $u \mapsto \A_\epsilon[u](u)$ and $u \mapsto \F_\epsilon(u)$ are Lipschitz continuous with
\begin{align} \label{eq:lipschitz}
\dprod{\F_\epsilon(u)-\F_\epsilon(v),w}=\A_\epsilon[u](u)(w)-\A_\epsilon[v](v)(w) \leq \sqrt{3} \xiep \ltn{\nabla(u-v)}\ltn{\nabla w},
\end{align}
$u,v,w \in X$, and strongly monotone with
\begin{align} \label{eq:smonotone}
\dprod{\F_\epsilon(u)-\F_\epsilon(v),u-v}=\A_\epsilon[u](u)(u-v)-\A_\epsilon[v](v)(u-v) \geq \xiem \ltn{\nabla(u-v)}^2, 
\end{align} 
$u,v \in X$.
\end{enumerate} 
\end{proposition}

\begin{proof}
Our proof is borrowed from~\cite{HeidSuli:2022}. First of all, by definition of $\A_\epsilon$, cf.~\eqref{eq:Aoperator}, the upper bound in~\eqref{eq:mubound}, and the Cauchy--Schwarz inequality we immediately obtain that
\begin{align*}
\A_\epsilon[u](v)(w) = \int_\Omega \mu_\epsilon(x,|\nabla u|^2) \nabla v \cdot \nabla w \dx \leq \muep \ltn{\nabla v} \ltn{\nabla w},
\end{align*}
which is the bound claimed in~\eqref{eq:Abd}. Similarly, by replacing the upper bound with the lower bound in~\eqref{eq:mubound}, we arrive at~\eqref{eq:coercive}.

Next, let $u,v,w \in X$ be arbitrarily. We have by definition of $\A_\epsilon$ that
\begin{align*}
\A_\epsilon[u](u)(w)-\A_\epsilon[v](v)(w)=\int_\Omega \left(\mu_\epsilon(x,|\nabla u|^2)\nabla u-\mu_\epsilon(x,|\nabla v|^2\right)\nabla v) \cdot \nabla w \dx. 
\end{align*}
Invoking~\eqref{eq:mukappaU} and the Cauchy--Schwarz inequality, this further implies
\begin{align*}
\A_\epsilon[u](u)(w)-\A_\epsilon[v](v)(w) &\leq \sqrt{3} \xiep \int_\Omega |\nabla(u-v)| | \nabla w| \dx\\
&= \sqrt{3} \xiep \ltn{\nabla (u-v)}\ltn{\nabla w}.
\end{align*} 
In a similar manner we find that, for any $u,v \in X$,
\begin{align*}
\A_\epsilon[u](u)(u-v)-\A_\epsilon[v](v)(u-v)=\int_\Omega \left(\mu_\epsilon(x,|\nabla u|^2)\nabla u-\mu_\epsilon(x,|\nabla v|^2\right)\nabla v) \cdot \nabla(u-v) \dx.
\end{align*}
As $\nabla(\cdot)$ is a linear operator, the bound~\eqref{eq:mukappal} implies the strong monotonicity~\eqref{eq:smonotone}.
\end{proof}

Since $\F_\epsilon$ is strongly monotone and Lipschitz continuous, we immediately obtain the existence and uniqueness result stated below; we refer to~\cite[\S 25.5]{Zeidler:90} for details.

\begin{theorem}
The energy functional $\E_\epsilon:X \to \mathbb{R}$ from~\eqref{eq:relaxedE} has a unique global energy minimiser $u_{X,\epsilon}^\star \in X$, which, equivalently, is the unique solution of the weak, relaxed problem~\eqref{eq:rwmuproblem}.
\end{theorem}

\section{Damped Ka\v{c}anov iteration scheme for the relaxed $p(x)$-Poisson equation} \label{sec:kacanov}

In this section, we will consider the damped Ka\v{c}anov scheme from~\cite{HeidWihlerm2an:2022}, which is defined as follows: For any given iterate $u^n \in X$, the ensuing element is given by
\begin{align} \label{eq:kacanov}
u^{n+1}=u^n -\delta(u^n)\A_\epsilon[u^n]^{-1}\F_\epsilon(u^n),
\end{align}
where $\delta(u^n)>0$ is a damping factor, and $\A_\epsilon[\cdot](\cdot): X \times X \to X'$ and $\F_\epsilon:X \to X'$ are defined as in~\eqref{eq:Aoperator} and~\eqref{eq:Fop}, respectively. We can equivalently state~\eqref{eq:kacanov} in the form
\begin{align} \label{eq:kacanov2}
\A_\epsilon[u^n](u^{n+1}-u^n)=-\delta(u^n)\F_\epsilon(u^n),
\end{align}
since, for fixed $u \in X$, the operator $\A_\epsilon[u]:X \to X'$ is linear. The latter formulation of the Ka\v{c}anov scheme will come in very handy in our analysis below. The goal of this section is to show that the sequence generated by the Ka\v{c}anov scheme~\eqref{eq:kacanov} converges to the unique solution $u_{X,\epsilon}^\star$ of our relaxed problem~\eqref{eq:weakrelaxed} in $X$; we emphasise once more that $u_{X,\epsilon}^\star$ is equivalently the unique minimiser of $\E_\epsilon$, cf.~\eqref{eq:relaxedE}, in $X$. For that purpose, we will need the following preliminary result, which, in particular, is borrowed from~\cite[Cor.~2.7]{HeidWihlerm2an:2022}. Even though our proof will largely proceed along the lines of the proof of~\cite[Cor.~2.7]{HeidWihlerm2an:2022}, we will include it in the work presented herein for the sake of completeness and since we use a distinct notation of the iteration scheme, which allows for some simplifications of the proof.

\begin{lemma}
Let $\{u^n\}_n \subset X$ denote the sequence generated by the Ka\v{c}anov scheme~\eqref{eq:kacanov}. Then we have that
\begin{align} \label{eq:energydiff}
\E(u^n)-\E(u^{n+1}) \geq \left(\frac{\muem}{\delta(u^n)}-\frac{\sqrt{3}\xiep}{2} \right)\nnnn{\nabla(u^{n+1}-u^n)}^2.
\end{align}
In particular, if $0<\delta_{\min} \leq \delta(u^n)\leq \delta_{\max}<\nicefrac{2 \muem}{\sqrt{3}\xiep}$ for all $n \in \mathbb{N}$, then 
\begin{align*} 
\E(u^n)-\E(u^{n+1}) \geq \gamma \nnnn{\nabla(u^{n+1}-u^n)}^2, \qquad n \in \mathbb{N},
\end{align*}
where 
\begin{align*} 
\gamma:=\left(\frac{\muem}{\delta_{\max}}-\frac{\sqrt{3} \xiep }{2}\right)>0.
\end{align*}
\end{lemma}

\begin{proof}
As mentioned before, we will proceed along the lines of the proof of~\cite[Cor.~2.7]{HeidWihlerm2an:2022}. For given $n \in \mathbb{N}$, let us define the real-valued function $\varphi(t):=\E_\epsilon(u^n+t(u^{n+1}-u^n))$. Then, the fundamental theorem of calculus implies that
\begin{align*}
\E_\epsilon(u^{n+1})-\E_\epsilon(u^n) &=\int_0^1 \dprod{\E_\epsilon'(u^n+t(u^{n+1}-u^n)),u^{n+1}-u^n} \dt\\
&= \int_0^1\dprod{\F_\epsilon(u^n+t(u^{n+1}-u^n))-\F_\epsilon(u^n),u^{n+1}-u^n} \dt+\dprod{\F_\epsilon(u^n),u^{n+1}-u^n},
\end{align*} 
where we used in the second step that $\E_\epsilon'=\F_\epsilon$. Next we may apply the Lipschitz continuity~\eqref{eq:lipschitz} and the definition of the Ka\v{c}anov iteration~\eqref{eq:kacanov2}, which yields
\begin{align*}
\E_\epsilon(u^{n+1})-\E_\epsilon(u^n) &\leq \frac{\sqrt{3} \xiep}{2} \nnnn{\nabla(u^{n+1}-u^n)}^2-\frac{1}{\delta(u^n)}\A_\epsilon[u^n](u^{n+1}-u^n)(u^{n+1}-u^n)\\
& \leq \frac{\sqrt{3} \xiep}{2} \nnnn{\nabla(u^{n+1}-u^n)}^2-\frac{\muem}{\delta(u^n)}\nnnn{\nabla(u^{n+1}-u^n)}^2
\end{align*}
thanks to~\eqref{eq:coercive}. Finally, a simple multiplication by $-1$ yields~\eqref{eq:energydiff}.
\end{proof}

Now we have all the ingredients to show that the damped Ka\v{c}anov iteration converges to the unique minimiser $\uex$ of $\E_\epsilon$ in $X$.

\begin{theorem}
If $\delta(u^n) \geq \delta_{\min}>0$ for all $n \in \mathbb{N}$ and
\begin{align} \label{eq:keypr}
\E_\epsilon(u^n)-\E_\epsilon(u^{n+1}) \geq \gamma \nnnn{\nabla(u^{n+1}-u^n)}^2,
\end{align}
for some $\gamma>0$ independent of $n \in \mathbb{N}$, then the sequence $\{u^n\}_n \subset X$ generated by the damped Ka\v{c}anov scheme~\eqref{eq:kacanov} converges the unique energy minimiser $\uex$ of $\E_\epsilon$ in $X$. Especially, this holds true if $0<\delta_{\min} \leq \delta(u^n)\leq \delta_{\max}<\nicefrac{2 \muem}{\sqrt{3}\xiep}$ for all $n \in \mathbb{N}$.
\end{theorem}

\begin{proof}
We will proceed along the lines of the proof of~\cite[Thm.~2.5]{HeidWihlerm2an:2022}, which is split into three parts.

\emph{Step 1:} Recall that $\E_\epsilon$ has a unique global minimiser $\uex$ in $X$ and that $\E_\epsilon$ decays along the generated sequence by our assumption~\eqref{eq:keypr}. Consequently, the sequence $\{\E_\epsilon(u^n)\}_n$ converges. Therefore, in light of~\eqref{eq:keypr}, we find that
\begin{align*}
0 \leq \nnnn{\nabla(u^{n+1}-u^n)}^2 \leq \frac{1}{\gamma}(\E_\epsilon(u^{n})-\E_\epsilon(u^{n+1})) \to 0 \qquad \text{as} \ n \to \infty;
\end{align*} 
i.e., $\left\{\nnnn{\nabla(u^{n+1}-u^n)}\right\}_n$ is a vanishing sequence. 

\emph{Step 2:} Next, we will verify that $\{u^n\}_n$ is a Cauchy sequence in $X$. Consequently, since $X$ is a Banach space, the sequence has a limit $\tilde{u} \in X$. The strong monotonicity of $\F_\epsilon$, cf.~\eqref{eq:smonotone}, yields that
\begin{multline*}
\xiem \nnnn{\nabla(u^m-u^n)}^2 \\
\begin{aligned} &\leq \dprod{\F_\epsilon(u^m)-\F_\epsilon(u^n),u^m-u^n} \\
&=\frac{1}{\delta(u^n)}\A_\epsilon[u^n](u^{n+1}-u^n)(u^m-u^n) -\frac{1}{\delta(u^m)}\A_\epsilon[u^m](u^{m+1}-u^m)(u^m-u^n), 
\end{aligned}
\end{multline*}
where we used~\eqref{eq:kacanov2} in the second step. Employing the uniform boundedness of $\A_\epsilon$, cf.~\eqref{eq:Abd}, we obtain
\begin{multline*}
\xiem\nnnn{\nabla(u^m-u^n)}^2 \\
\leq \frac{\muep}{\delta_{\min}} \left(\nnnn{\nabla(u^{n+1}-u^n)}+\nnnn{\nabla(u^{m+1}-u^m)}\right)\nnnn{\nabla(u^m-u^n)},
\end{multline*}
and thus by a simple manipulation
\begin{align*}
0 \leq \nnnn{\nabla(u^m-u^n)} \leq \frac{\muep}{\xiem \delta_{\min}}\left(\nnnn{\nabla(u^{n+1}-u^n)}+\nnnn{\nabla(u^{m+1}-u^m)}\right).
\end{align*}
Thanks to the first step, the right-hand side goes to zero as $m,n \to \infty$. Therefore, $\{u^n\}_n$ is indeed a Cauchy sequence.

\emph{Step 3:} It remains to show that $\tilde{u}=\uex$. To that end, we first recall~\eqref{eq:kacanov2}, which states that
\begin{align*}
\frac{1}{\delta(u^n)}\A_\epsilon[u^n](u^{n+1}-u^n)(w)=\dprod{\F_\epsilon(u^n),w} \qquad \text{for all} \ w \in X.
\end{align*}
Since $\delta(\cdot)\geq \delta_{\min}>0$ and the uniform boundedness of $\A_\epsilon$, cf.~\eqref{eq:Abd}, Step 1 yields that, for fixed $w \in X$, the left-hand side vanishes as $n \to \infty$. Therefore, we have for any $w \in X$ that
\begin{align*}
0=\lim_{n \to \infty} \dprod{\F_\epsilon(u^n),w}= \dprod{\F_\epsilon(\tilde{u}),w},
\end{align*}
where we used the continuity of $\F_\epsilon$ in the second equality. This, however, means that $\tilde{u}$ is a solution of the weak problem~\eqref{eq:weakrelaxed} on $X$. In turn, by the uniqueness of the solution, we have that $\tilde{u}=\uex$, which concludes the proof.
\end{proof}

\section{Convergence with respect to the relaxation and discretisation}

\subsection{Convergence with respect to the relaxation on discrete spaces} Let $X_N$ be a closed, finite dimensional subset of $W^{1,\infty}_0(\Omega) \subset \hs$. Note that, thanks to Proposition~\ref{eq:embedding}, we further have $X_N \subset \wop$. We will denote by $u_N^\star \in X_N$ the unique discrete solution of~\eqref{eq:weakpx} in $X_N$, and, for a fixed pair of cut-off parameters, $u_{N,\epsilon}^\star \in X_N$ signifies the unique solution of the discrete, relaxed problem~\eqref{eq:rwmuproblem} in $X_N$. In the sequel, we write $\epsilon \to \infty$ if $\epsm \to 0$ and $\epsp \to \infty$. We note that the set of cut-off parameters, which shall be denoted by 
\[\mathcal{A}:=\{(\epsm,\epsp):0<\epsm<1<\epsp<\infty\},\]
is a directed set with the ordering being given by the inclusion of the corresponding closed intervals; i.e., $\epsilon_1 \leq \epsilon_2$ if and only if $[\epsilon_{1,-},\epsilon_{1,+}] \subseteq [\epsilon_{2,-},\epsilon_{2,+}]$. Our goal is to prove that $\une \to \un$ as $\epsilon \to \infty$, for which end we need the following auxiliary result.

\begin{lemma} \label{lem:helpeps}
Let $\{v_\epsilon\}_\epsilon \subset X_N$ be uniformly bounded with respect to any norm on $X_N$. Then, we have that
\begin{align*}
\lim_{\epsilon \to \infty} \left|\E(v_\epsilon)-\E_\epsilon(v_\epsilon)\right|=0,
\end{align*}
where the convergence is in the sense of a net.
\end{lemma}

\begin{proof}
By definition of the operators $\E$ and $\E_\epsilon$, see~\eqref{eq:penergy} and~\eqref{eq:relaxedE}, respectively, we have that 
\begin{align*}
\left|\E(v_\epsilon)-\E_\epsilon(v_\epsilon)\right|=\left|\int_\Omega \frac{1}{p(x)} |\nabla v_\epsilon|^{p(x)} -\varphi_\epsilon(x,|\nabla v_\epsilon|^2) \dx\right|.
\end{align*}
Since the set $\{v_\epsilon\}_\epsilon$ is uniformly bounded in $X_N$, and by the equivalence of norms in finite dimensional spaces, we have that $\norm{\nabla v_\epsilon}_{L^{\infty}(\Omega)} \leq C$ for some positive constant $C>0$ independent of $\epsilon$. In turn, for $\epsilon$ big enough we have that $\norm{\nabla v_\epsilon}_{L^{\infty}(\Omega)} < \epsp$, and thus the upper cut-off parameter can be neglected in the limit. Consequently, a straightforward calculation reveals that, for $\epsilon$ big enough,
\begin{align*}
\left|\E(v_\epsilon)-\E_\epsilon(v_\epsilon)\right| &= \left|\int_{\Omega_{\epsilon,-}} \frac{1}{p(x)} |\nabla v_\epsilon|^{p(x)}-\frac{1}{2}|\nabla v_\epsilon|^{2}\epsm^{p(x)-2}-\left(\frac{1}{p(x)}-\frac{1}{2}\right)\epsm^{p(x)} \dx\right|\\
& \leq \int_{\Omega_{\epsilon,-}} \frac{1}{p(x)}|\nabla v_\epsilon|^{p(x)} +\frac{1}{2} |\nabla v_\epsilon|^{2}\epsm^{p(x)-2}+\left|\frac{1}{p(x)}-\frac{1}{2}\right| \epsm^{p(x)} \dx,
\end{align*} 
where $\Omega_{\epsilon,-}:=\{x \in \Omega: |\nabla v_\epsilon|< \epsm\}$. In turn, since $0<\epsm<1$ and $p(x)>1$ for all $x \in \Omega$, we immediately find that
\begin{align} \label{eq:bound1}
\left|\E(v_\epsilon)-\E_\epsilon(v_\epsilon)\right| \leq \left(\frac{1}{p_{-}}+1\right) \epsm^{p_{-}} |\Omega|,
\end{align}
with $|\Omega|$ signifying the Lebesgue measure of $\Omega$. Hence, if $\epsilon \to \infty$, and in turn $\epsilon_{-} \to 0$, the right-hand side in~\eqref{eq:bound1} vanishes, which proves the claim.
\end{proof}

This result paves the way for the proof of the following convergence statement.

\begin{theorem} \label{thm:relaxationconvergence}
Let $f$ satisfy {\rm ($A_\alpha$)} with \eqref{eq:alpahmin}. Then, for any given closed, finite dimensional subspace $X_N \subset W^{1,\infty}_{0}(\Omega)$, we have that 
\begin{align*}
\lim_{\epsilon \to \infty} \E_\epsilon(u_{\epsilon,N}^\star)=\E(u_{N}^\star),
\end{align*}
where $u_{\epsilon,N}^\star$ and $u_N^\star$ are the unique discrete solutions of~\eqref{eq:rwmuproblem} and~\eqref{eq:weakpx} in $X_N$, respectively.
\end{theorem}

\begin{proof}
Lemma~\ref{lem:helpeps}, for $v_\epsilon \equiv u_N^\star$ independent of $\epsilon$, immediately implies that $\lim_{\epsilon \to \infty} \E_\epsilon(u_N^\star)=\E(u_N^\star)$. Moreover, as $u_N^\star$ is the global minimiser of $\E$ in $X_N$, we further find that
\begin{align} \label{eq:helpthm1}
\liminf_{\epsilon \to \infty} \E(u_{\epsilon,N}^\star) \geq \E(u_N^\star)=\limsup_{\epsilon \to \infty} \E_\epsilon(u_N^\star) \geq \limsup_{\epsilon \to \infty} \E_\epsilon(u_{\epsilon,N}^\star),
\end{align}
where we used in the last inequality that $u_{\epsilon,N}^\star$ is the global minimiser of $\E_\epsilon$ in $X_N$. Consequently, if we can show that 
\begin{align*}
\liminf_{\epsilon \to \infty} \E(u_{\epsilon,N}^\star)=\liminf_{\epsilon \to \infty} \E_\epsilon(u_{\epsilon,N}^\star),
\end{align*}
then~\eqref{eq:helpthm1} implies the claim. In particular, in light of Lemma~\ref{lem:convergence}, it is only left to verify that the net $\{u_{\epsilon,N}^\star\}_{\epsilon} \subset X_N$ is uniformly bounded. So take any $\epsilon \in \mathcal{A}$, i.e., a pair of cut-off parameters with $0<\epsm<1<\epsp<\infty$. By assumption, $u_{\epsilon,N}^\star$ is a solution of the relaxed, discrete problem~\eqref{eq:rwmuproblem}. Therefore, and thanks to H\"{o}lders inequality, we have that
\begin{multline} \label{eq:RHSbound}
\int_\Omega \mu_\epsilon(x,|\nabla u_{N,\epsilon}^\star|^2)|\nabla u_{N,\epsilon}^\star|^2 \dx = \int_\Omega f u_{N,\epsilon}^\star \dx \\
\leq \norm{f}_{L^{\nicefrac{6}{5}}(\Omega)} \norm{u_{N,\epsilon}^\star}_{L^6(\Omega)} \leq C_1 \norm{f}_{L^{\nicefrac{6}{5}}(\Omega)} \norm{\nabla u_{N,\epsilon}^\star}_{L^2(\Omega)};
\end{multline}
here, $C_1>0$ is the constant from the continuous embedding $\hs \hookrightarrow L^6(\Omega)$, which is independent of $\epsilon$ and $X_N$. To determine a lower bound of the left-hand side in~\eqref{eq:RHSbound}, we introduce the sets
\begin{align*} 
\Omega_{\epsilon,-}:=\{x \in \Omega: |\nabla u_{N,\epsilon}^\star| < 1\} \quad \text{and} \quad \Omega_{\epsilon,+}:=\{x \in \Omega: |\nabla u_{N,\epsilon}^\star| \geq 1\}.
\end{align*}
Then a straightforward calculation reveals that 
\begin{align*}
\mu_\epsilon(x,|\nabla u_{N,\epsilon}^\star|^2)|\nabla u_{N,\epsilon}^\star|^2 \geq |\nabla u_{N,\epsilon}^\star|^{\re} \qquad \text{for all} \ x \in \omegamin
\end{align*}
and 
\begin{align*}
\mu_\epsilon(x,|\nabla u_{N,\epsilon}^\star|^2)|\nabla u_{N,\epsilon}^\star|^2 \geq |\nabla u_{N,\epsilon}^\star|^{\se} \qquad \text{for all} \ x \in \omegaplus,
\end{align*}
where 
\begin{align*}
(\se,\re):=
\begin{cases}
(p_{-},2) & 1<p_{-} \leq p_{+} < 2, \\
(p_{-},p_{+}) & 1<p_{-} \leq 2 \leq p_{+} < \infty, \\
(2,p_{+}) & 2 < p_{-} \leq p_{+} < \infty.
\end{cases}
\end{align*}
As a consequence, we have
\begin{align} \label{eq:LHSbound}
\int_\Omega \mu_\epsilon(x,|\nabla u_{N,\epsilon}^\star|^2)|\nabla u_{N,\epsilon}^\star|^2 \dx \geq \norm{\nabla u_{N,\epsilon}^\star}_{L^{\re}(\omegamin)}^{\re}+\norm{\nabla u_{N,\epsilon}^\star}_{L^{\se}(\omegaplus)}^{\se}.
\end{align}
In the following, we shall distinguish two cases:
\begin{enumerate}[(a)]
\item If $\norm{\nabla u_{N,\epsilon}^\star}_{L^{\se}(\omegamin)} \geq \norm{\nabla u_{N,\epsilon}^\star}_{L^{\se}(\omegaplus)}$, then, by definition of the set $\omegamin$,
\begin{align*}
\norm{\nabla u_{N,\epsilon}^\star}_{L^{\se}(\omegaplus)}^{\se} \leq \int_{\omegamin} |\nabla u_{N,\epsilon}^\star|^{\se} \dx \leq |\omegamin| \leq |\Omega|,
\end{align*}
and, in turn,
\begin{align*}
\norm{\nabla u_{N,\epsilon}^\star}_{L^{\se}(\Omega)}^{\se}=\norm{\nabla u_{N,\epsilon}^\star}_{L^{\se}(\omegaplus)}^{\se}+\norm{\nabla u_{N,\epsilon}^\star}_{L^{\se}(\omegamin)}^{\se} \leq 2 |\Omega|.
\end{align*}
In particular, we obtained the uniform bound
\begin{align*}
\norm{\nabla u_{N,\epsilon}^\star}_{L^{\se}(\Omega)} \leq (2 |\Omega|)^{\nicefrac{1}{\se}}.
\end{align*}
\item Now let us assume that $\norm{\nabla u_{N,\epsilon}^\star}_{L^{\se}(\omegamin)} \leq \norm{\nabla u_{N,\epsilon}^\star}_{L^{\se}(\omegaplus)}$. Then, in combination with the equivalence of norms in the finite dimensional space $X_N$, we find that
\begin{align*}
\norm{\nabla u_{N,\epsilon}^\star}_{L^2(\Omega)} \leq C_2 \norm{\nabla u_{N,\epsilon}^\star}_{L^{\se}(\Omega)} \leq 2^{\nicefrac{1}{\se}} C_2 \norm{\nabla u_{N,\epsilon}^\star}_{L^{\se}(\omegaplus)},
\end{align*}
for some constant $C_2>0$ independent of $\epsilon$ (but, possibly, depending on the specific finite dimensional space $X_N$). In combination with~\eqref{eq:RHSbound} this yields
\begin{align} \label{eq:finalrhsbound}
\int_\Omega \mu_\epsilon(x,|\nabla u_{N,\epsilon}^\star|^2)|\nabla u_{N,\epsilon}^\star|^2 \dx \leq 2^{\nicefrac{1}{\se}}C_1C_2 \norm{f}_{L^{\nicefrac{6}{5}}(\Omega)} \norm{\nabla u_{N,\epsilon}^\star}_{L^{\se}(\omegaplus)}.
\end{align}
On the other hand we have, thanks to ~\eqref{eq:LHSbound},
\begin{align} \label{eq:finallhsbound}
\int_\Omega \mu_\epsilon(x,|\nabla u_{N,\epsilon}^\star|^2)|\nabla u_{N,\epsilon}^\star|^2 \dx \geq \norm{\nabla u_{N,\epsilon}^\star}_{L^{\se}(\omegaplus)}^{\se}
\end{align}
Combining \eqref{eq:finalrhsbound} and~\eqref{eq:finallhsbound} yields
\begin{align*}
\norm{\nabla u_{N,\epsilon}^\star}_{L^{\se}(\omegaplus)} \leq \left(2^{\nicefrac{1}{\se}}C_1C_2 \norm{f}_{L^{\nicefrac{6}{5}}(\Omega)} \right)^{\nicefrac{1}{(\se-1)}},
\end{align*}
and, in turn,
\begin{align*} 
\norm{\nabla u_{N,\epsilon}^\star}_{L^{\se}(\Omega)} \leq 2^{\nicefrac{1}{\se}}\left(2^{\nicefrac{1}{\se}}C_1C_2 \norm{f}_{L^{\nicefrac{6}{5}}(\Omega)} \right)^{\nicefrac{1}{(\se-1)}};
\end{align*}
i.e., we found a uniform bound.
\end{enumerate}
Since (a) and (b) cover all the possible cases, we verified that $\{u_{\epsilon,N}^\star\}_{\epsilon}$ is indeed uniformly bounded (with respect to the $W_0^{1,\se}$-norm), which concludes the proof.
\end{proof}

\subsection{Convergence with respect to a sequence of hierarchical discrete spaces} 
Now let us assume that we have a sequence of hierarchical discrete spaces $\{X_N\}_N$ such that $X_N \subset W^{1,\infty}_0(\Omega)$ for all $N \in \mathbb{N}$, and $X_0 \subset X_1 \subset \dotsc \subset X_N \subset X_{N+1} \subset \dotsc W^{1,\infty}_0(\Omega)$. We aim to derive the convergence of the discrete solutions $u_N^\star$ of~\eqref{eq:weakpx} in $X_N$, for $N \in \mathbb{N}$, to the continuous solution $u^\star \in \wop$. For that purpose, we need the following preliminary result.

\begin{lemma} \label{lem:convergence}
If $v_N,v \in \wop$ with $\lim_{N \to \infty}\norm{\nabla(v_N-v)}_{L^{p(\cdot)}(\Omega)}=0$, then, for a suitable subsequence $\{v_{N_j}\}_j$,
\begin{align} \label{eq:pintcon}
\lim_{j \to \infty} \int_\Omega \frac{1}{p(x)}\left(|\nabla v_{N_j}|^{p(x)}-|\nabla v|^{p(x)}\right) \dx =0.
\end{align}
\end{lemma}

\begin{proof}
A simple manipulation reveals that
\begin{align} \label{eq:firstin}
0 \leq \left|\int_\Omega \frac{1}{p(x)}\left(|\nabla v_{N}|^{p(x)}-|\nabla v|^{p(x)}\right) \dx\right| \leq \frac{1}{p_{-}}\int_\Omega \left||\nabla v_N|^{p(x)}-|\nabla v|^{p(x)}\right| \dx,
\end{align}
thus it suffices to show that the integral on the right-hand side vanishes for a subsequence $\{N_j\}_j$. To that end, we first note that 
\begin{align*} 
0 \leq \left||\nabla v_N|^{p(x)}-|\nabla v|^{p(x)}\right| \leq |\nabla v_N|^{p(x)}+|\nabla v|^{p(x)}=:g(x) \in L^1(\Omega).
\end{align*}
Furthermore, in light of Proposition~\ref{prop:equiv}, the assumption on the convergence in $\wop$ implies that $|\nabla v_N|$ converges in measure to $|\nabla v|$. In turn, there exists a subsequence $\{v_{N_j}\}_j$ so that $|\nabla v_{N_j}|$ converges almost everywhere to $|\nabla v|$; we refer, e.g., to~\cite[Prop.~3.1.3]{Cohn:2013}. In particular, we have that
\begin{align*} 
\left||\nabla v_{N_j}(x)|^{p(x)}-|\nabla v(x)|^{p(x)}\right| \to 0 \qquad \text{for almost every} \ x \in \Omega.
\end{align*}
Consequently, the dominated convergence theorem implies that 
\begin{align*}
\int_\Omega \left||\nabla v_{N_j}|^{p(x)}-|\nabla v|^{p(x)}\right| \dx \to 0 \quad \text{as} \ j \to \infty,
\end{align*}
which, together with~\eqref{eq:firstin}, proves the claim.
\end{proof}

\begin{theorem} \label{thm:convergenceinN}
Assume that the union of the discrete spaces $X_N$, $N \in \mathbb{N}$, is dense in $\wop$; i.e.,
\begin{align} \label{eq:density}
\overline{\bigcup_{N \in \mathbb{N}} X_N}^{\norm{\cdot}_{W_0^{1,p(\cdot)}(\Omega)}}=W_0^{1,p(\cdot)}(\Omega).
\end{align}
Then, if $f$ satisfies {\rm ($A_\alpha$)} with \eqref{eq:alpahmin}, we have that $\E(u_N^\star) \to \E(u^\star)$ as $N \to \infty$. 
\end{theorem}

\begin{proof}
We want to show that 
\begin{align} \label{eq:thm1}
0 \leq \E(u_N^\star)-\E(u^\star) \to 0 \quad \text{as} \ N \to \infty;
\end{align}
the first inequality holds true since $u^\star$ is the global minimiser of $\E$ in $\wop$ and  $u_N^\star \in W_0^{1,\infty}(\Omega) \subset \wop$ for all $N \in \mathbb{N}$. Moreover, since $\un$ is the global minimiser of $\E$ in $X_N$ and by the nestedness of the discrete spaces, i.e., $X_N \subset X_{N+1}$ for all $N \in \mathbb{N}$, we further have that the difference $\E(u_N^\star)-\E(u^\star)$ is decreasing. As a consequence, it is sufficient to verify~\eqref{eq:thm1} for any subsequence of $\{u_N^\star\}_N$. Next we note that the density~\eqref{eq:density}, in combination with the nestedness of the discrete spaces, yields the existence of a sequence $\mathfrak{u}_N \in X_N$, $N \in \mathbb{N}$, such that
\begin{align} \label{eq:convergencep}
\norm{\mathfrak{u}_N-u^\star}_{\wop} \to 0 \quad \text{as} \ N \to \infty.
\end{align}
Hence, we may exploit Lemma~\ref{lem:convergence}; let us denote by $\{\mathfrak{u}_{N_j}\}_j$ the corresponding subsequence satisfying~\eqref{eq:pintcon}. Since $u_{N_j}^\star$ is the global energy minimiser of $\E$ in $X_{N_j}$, it holds that
\begin{multline*} 
0 \leq \E(u_{N_j}^\star)-\E(u^\star) \leq \E(\mathfrak{u}_{N_j})-\E(u^\star) \\
= \int_\Omega \frac{1}{p(x)} \left(|\nabla \mathfrak{u}_{N_j}|^{p(x)}-|\nabla u^\star|^{p(x)} \right) \dx + \int_\Omega f (\mathfrak{u}_{N_j}-u^\star) \dx. 
\end{multline*}
The first term on the right-hand side above vanishes as $j \to \infty$ since $\{\mathfrak{u}_{N_j}\}_j$ satisfies~\eqref{eq:pintcon}. It remains to verify that the same holds true for the second term. For that purpose we first note that for $s \in C_{+}(\overline{\Omega})$ given by
\begin{align*}
\frac{1}{p^\star(x)}+\frac{1}{\alpha(x)}=\frac{1}{s(x)}, \qquad x \in \Omega
\end{align*}
we have $s(x) \geq 1$. Therefore, invoking Propositions~\ref{eq:embedding} and~\ref{eq:holder} leads to
\begin{align*}
\int_\Omega f (\mathfrak{u}_{N_j}-u^\star) \dx \leq C \norm{f(\mathfrak{u}_{N_j}-u^\star)}_{L^{s(\cdot)}(\Omega)} \leq 2 C \norm{f}_{L^{\alpha(\cdot)}(\Omega)} \norm{\mathfrak{u}_{N_j}-u^\star}_{L^{p^\star(\cdot)}(\Omega)},
\end{align*}
for a constant $C>0$ independent of $N_j$. Thanks to Proposition~\ref{prop:embedding}, we further obtain that
\begin{align} \label{eq:72}
\int_\Omega f (\mathfrak{u}_{N_j}-u^\star) \dx \leq 2 \widetilde{C} \norm{f}_{L^{\alpha(\cdot)}(\Omega)} \norm{\nabla(\mathfrak{u}_{N_j}-u^\star)}_{L^{p(\cdot)}(\Omega)},
\end{align}
where $\tilde{C}>0$ is still a positive constant independent of $N_j$. Thus the right-hand side of~\eqref{eq:72} vanishes thanks to~\eqref{eq:convergencep}; this concludes the proof.
\end{proof}



\section{Numerical experiments}

In this section, we will perform some numerical experiments to assess our theoretical findings. More specifically, we want to numerically investigate the convergence with respect to the number of damped Ka\v{c}anov iteration steps, the relaxation parameter, as well as the mesh size. For the construction of our discrete subspaces, we will consider a $P1$-finite element scheme. In particular, we consider a sequence of shape-regular conforming triangulations of $\Omega$, such that $\mathcal{T}_{n+1}$ is obtained by a refinement of $\mathcal{T}_{n}$. Then, the corresponding conforming finite element spaces are given by 
\begin{align*} 
X_N:=\left\{u \in W_0^{1,\infty}(\Omega):u|_K \in \mathbb{P}_1(K) \ \text{for all} \ K \in \mathcal{T}_N\right\}, 
\end{align*}
where $\mathbb{P}_1(K)$ signifies the set of all polynomials of degree at most one on $K$.

\begin{remark}
We note that in the given $P1$-FEM setting, the assumption~\eqref{eq:density} is satisfied if the mesh size of $\mathcal{T}_N$ goes to 0.
\end{remark}

We will consider the two model equations  
\begin{align} \label{eq:modelproblem}
\text{(MEQ.~i)} \quad \int_{\Omega_i} |\nabla u(\x)|^{p_i(\x)-2} \nabla u(\x) \cdot \nabla v(\x) \dd \x=\int f_i(\x) v(\x) \dd \x \quad \text{for all} \ v \in W_0^{1,p_i(\x)}(\Omega_i),
\end{align}
where
\begin{itemize}
\item $\Omega_1=(-1,1)^2 \subset \mathbb{R}^2$ and $p_1(x,y)=2.3+0.5x+0.5y$,
\item $\Omega_2=(0,1)^2 \subset \mathbb{R}^2$ and $p_2(x,y)=1.2+2(x^2+y^2)$;
\end{itemize}
here, $\x=(x,y) \in \mathbb{R}^2$ denote the Euclidean coordinates. In both cases, the source $f_i: \Omega \to \mathbb{R}$ is chosen in such a way that the exact solution of~\eqref{eq:modelproblem} is given by $u^\star(x,y)=\sin(\pi x) \sin(\pi y)$. We note that the variable exponent $p_i$, for $i \in \{1,2\}$, attains values below and above the threshold value $p=2$. Furthermore, the damping parameter for the Ka\v{c}anov scheme was chosen (fixed and smaller than one) in such a way that the iteration scheme converged, but without fine tuning. 

\begin{experiment}[Convergence of the damped Ka\v{c}anov scheme] \label{exp:kacanov}
In our first experiment, we will investigate the convergence of the damped Ka\v{c}anov iteration for the discrete, relaxed problem corresponding to~\eqref{eq:modelproblem}, for $i \in \{1,2\}$. Thereby, the discretisation is based on a triangular mesh with $\mathcal{O}(10^5)$ elements and the relaxation parameters are given by $\epsm=10^{-6}$ and $\epsp=10^6$. In both cases, the discrete, relaxed solution, which we will refer to as the reference solution $u_{\mathrm{ref},i}$, was approximated by 300 iteration steps of the damped Ka\v{c}anov scheme~\eqref{eq:kacanov} with the initial guess being the linear interpolant of the exact solution (of the unrelaxed problem) in the element nodes. Subsequently, we examined the decay of the error $\nnnn{\nabla(u_{\mathrm{ref},i}-u_i^n)}$, where $u_i^n$ is the $n$th iterate of the damped Ka\v{c}anov scheme~\eqref{eq:kacanov} for the discrete, relaxed version of (MEQ.~i), $i \in \{1,2\}$. The initial guess was chosen to be $u_1^0 \equiv 0$ and $u_2^0=\sin(\pi x y)$, or more precisely the linear interpolant in the mesh nodes, respectively. We emphasise that the initial guess in the former case is very unconsidered, as this leads to $\mu(\x,|\nabla u_1^0|^2) = \epsm^{p_1(\x)-2}$, which, depending on $\x \in \Omega$, can be extremely small or large. In turn, the same holds true for the entries of the matrix corresponding to the bilinear form $\A_\epsilon[u_1^0]$. This is, most probably, the reason why the first iteration step leads to an enormous error $\nnnn{\nabla(u_{\mathrm{ref},1}-u_1^1)}$. Nonetheless, after the initial step, we have a nice decay of the error, see Figure~\ref{fig:kacanov} (left). Also in the second case, i.e., for $i=2$, we have a neat convergence of the error to zero for an increasing number of iteration steps, see Figure~\ref{fig:kacanov} (right). Moreover, thanks to a wiser choice of the initial guess, we have a much smaller error in the initial phase compared to before. 

\begin{figure}[ht]
{\includegraphics[width=0.48\textwidth]{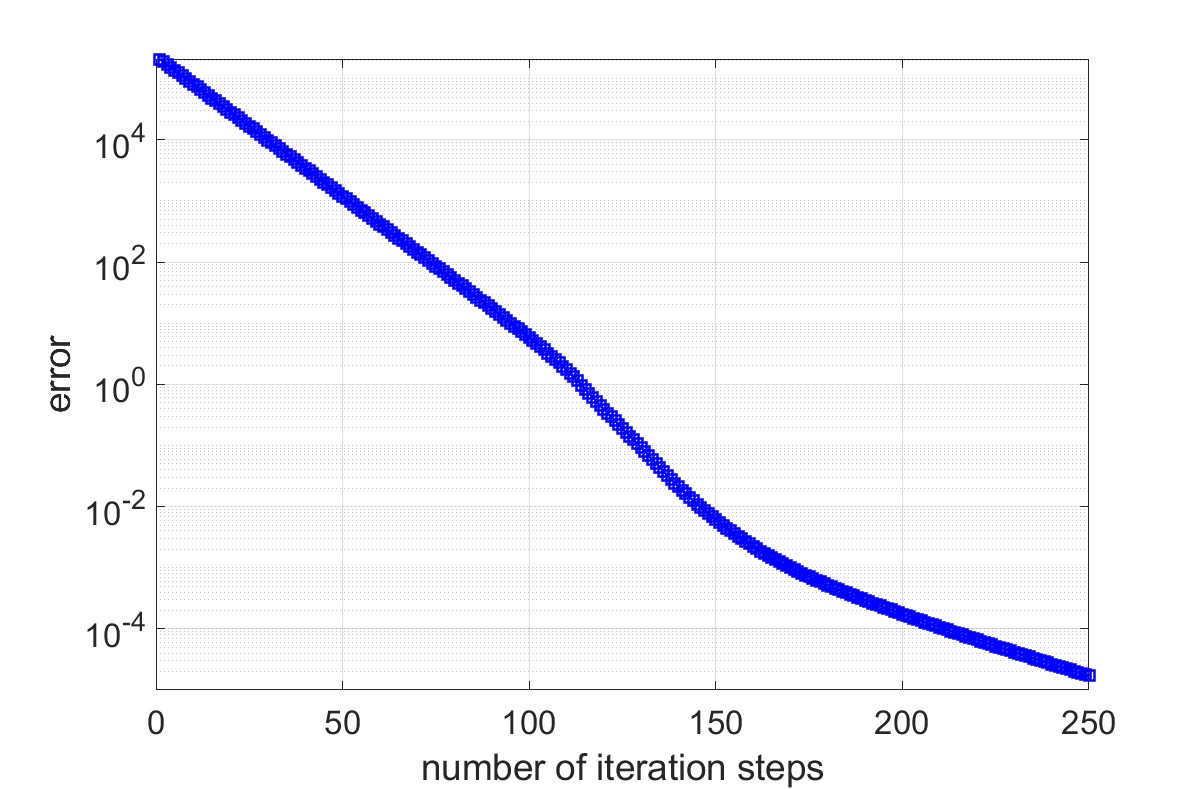}}\hfill
{\includegraphics[width=0.48\textwidth]{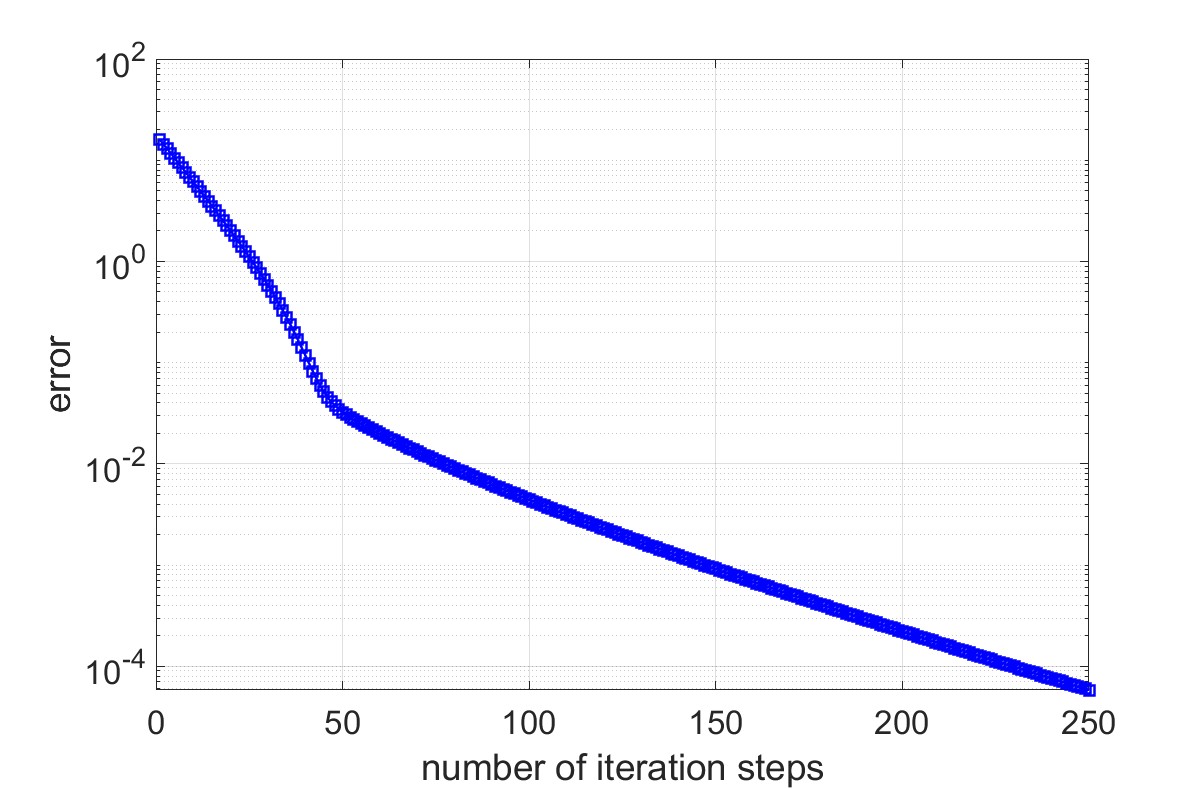}}
 \caption{Experiment~\ref{exp:kacanov}. Convergence of the damped Ka\v{c}anov scheme for the discrete, relaxed problem corresponding to (MEQ.~1) (left) and (MEQ.~2) (right), respectively.}\label{fig:kacanov}
\end{figure}
\end{experiment}

\begin{experiment}[Convergence with respect to the relaxation parameters] \label{exp:relaxation}
Next, we shall examine the convergence of the solution of the discrete, relaxed problem to the solution of discrete, unrelaxed problem. We use the same finite element space as before, and, in addition, we will consider the reference solutions from the previous experiment as our reference solutions in the given experiment. Here, for $i \in \{1,2\}$, we first compute approximations $u_{k,i}$ of the solutions of the relaxed, discrete versions of~\eqref{eq:modelproblem} for the relaxation parameters $\epsilon_{\pm}=1.4^{\pm k}$. For that purpose, for fixed $k$, we apply the damped Ka\v{c}anov scheme and stop our calculation as soon as the error, measured in the $\hs$-norm, of two consecutive iterates drops below $10^{-10}$. Subsequently, we plot the error $\nnnn{\nabla(u_{i,k}-u_{\mathrm{ref}})}$ against the exponent $k$ of the relaxation parameters. As predicted by our theory (cf.~Theorem~\ref{thm:relaxationconvergence}) --- albeit we consider here a stronger notion of convergence ---  the error nicely decays for an increasing $k$, see Figure~\ref{fig:relaxation}.

\begin{figure}[ht] 
{\includegraphics[width=0.48\textwidth]{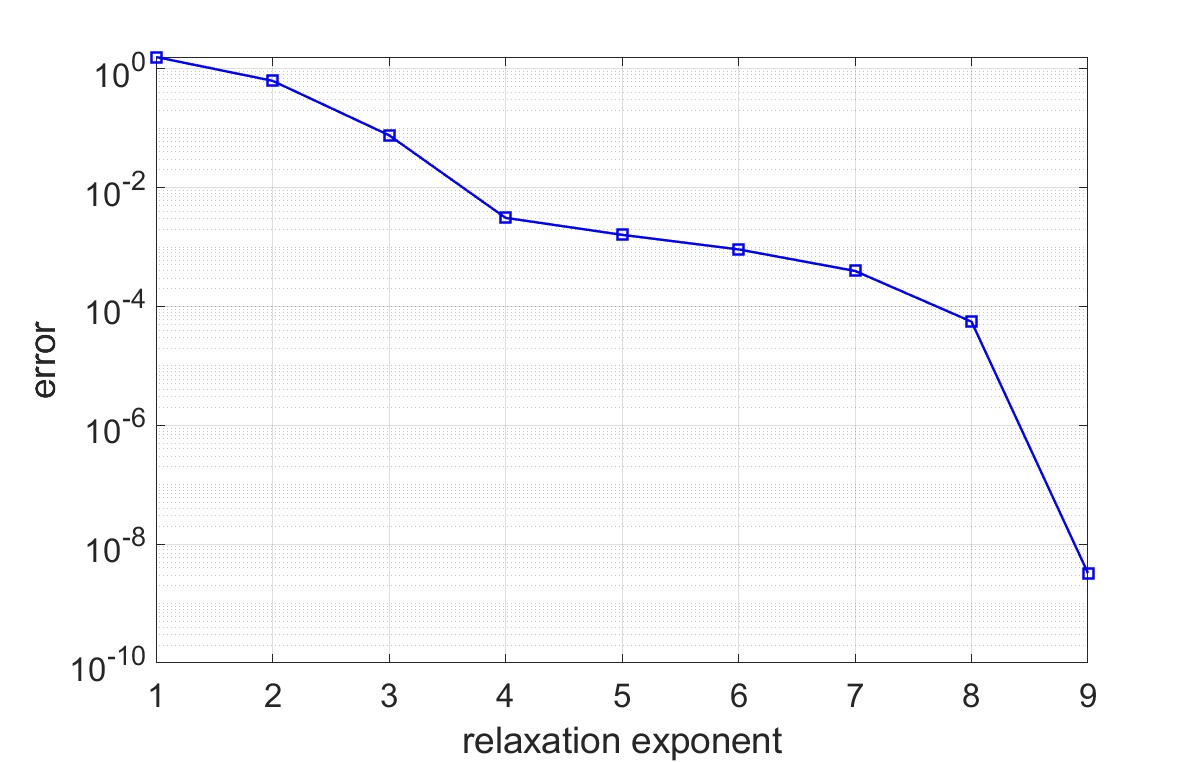}}\hfill
{\includegraphics[width=0.48\textwidth]{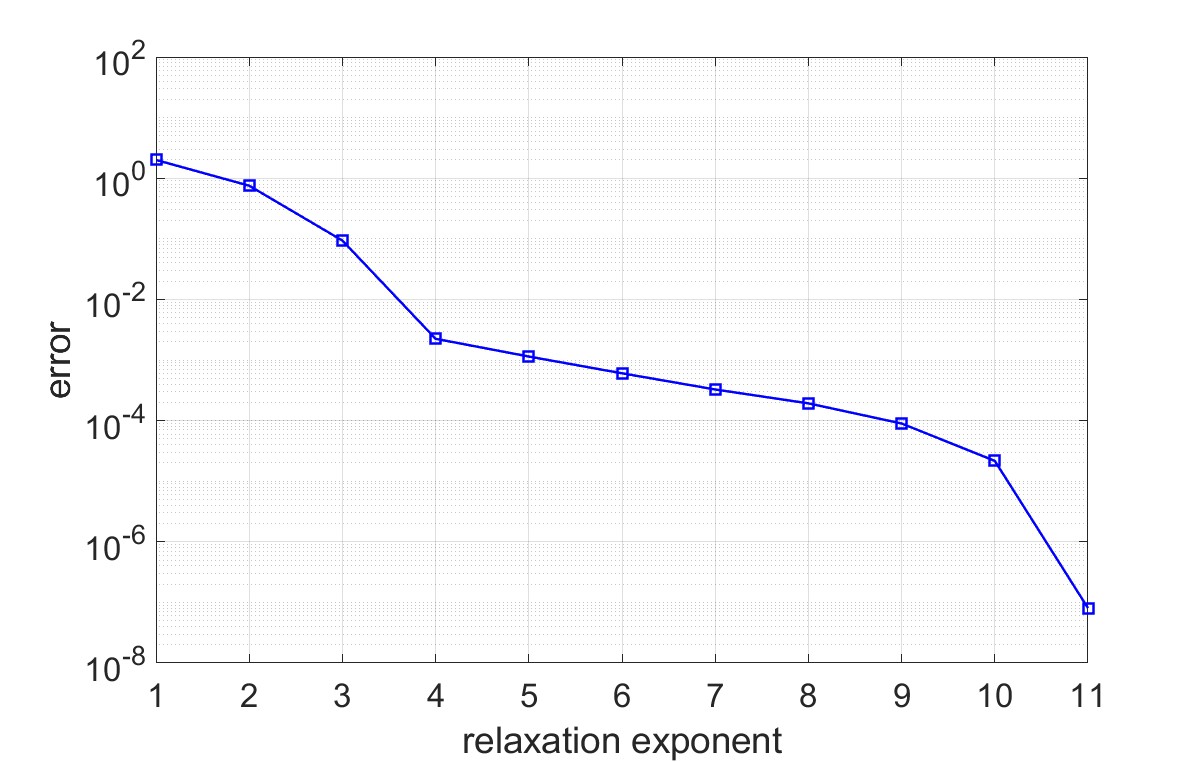}}
 \caption{Experiment~\ref{exp:relaxation}. Convergence with respect to the continuation parameters $\epsilon_{\pm}=1.4^{\pm k}$ for our model problem (MEQ.~1) (left) and (MEQ.~2) (right), respectively.}\label{fig:relaxation}
\end{figure}
\end{experiment}

\begin{experiment}[Convergence with respect to the mesh size] \label{exp:mesh}
In our last experiment, we are interested in the convergence with respect to the mesh size. For each given finite element space $X_N$, we approximate the solution $u_N^\star$ of the discrete, unrelaxed problem by applying the damped Ka\v{c}anov scheme, with the same stopping strategy as before, for the relaxed problem with relaxation parameters $\epsilon_{\pm}=10^{-6}$. We start this experiment with a coarse, uniform mesh $\mathcal{T}_0$, and employ a uniform mesh refinement to obtain $\mathcal{T}_{N+1}$ from $\mathcal{T}_N$. In Figure~\ref{fig:mesh} we depict the convergence of the error $\nnnn{\nabla(u_N^\star-u^\star)}$ against the number of elements in the mesh $\mathcal{T}_N$; here, we indeed consider the exact solution $u^\star(x,y)=\sin(\pi x) \sin(\pi y)$ of our model problem (MEQ.~i), for $i \in \{1,2\}$; cf.~\eqref{eq:modelproblem}. As can be observed in Figure~\ref{fig:mesh}, we have a linear decay of the error in both cases. 

\begin{figure}[ht] 
{\includegraphics[width=0.48\textwidth]{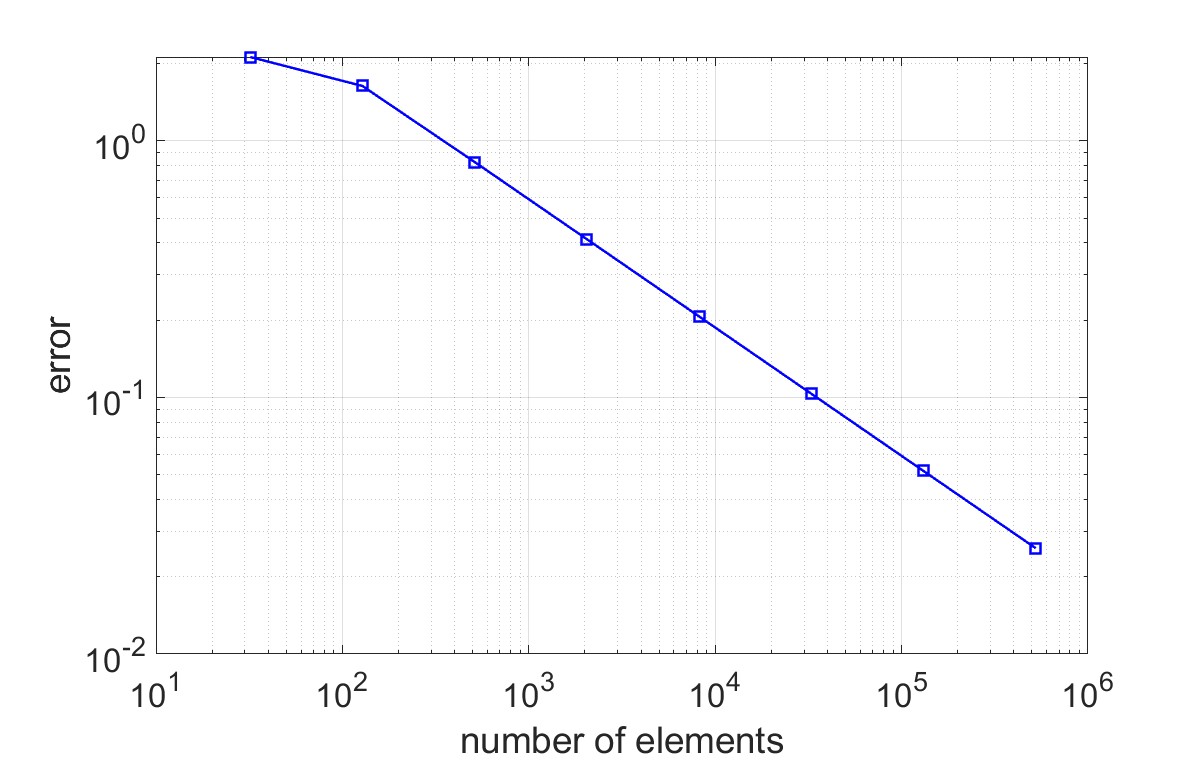}}\hfill
{\includegraphics[width=0.48\textwidth]{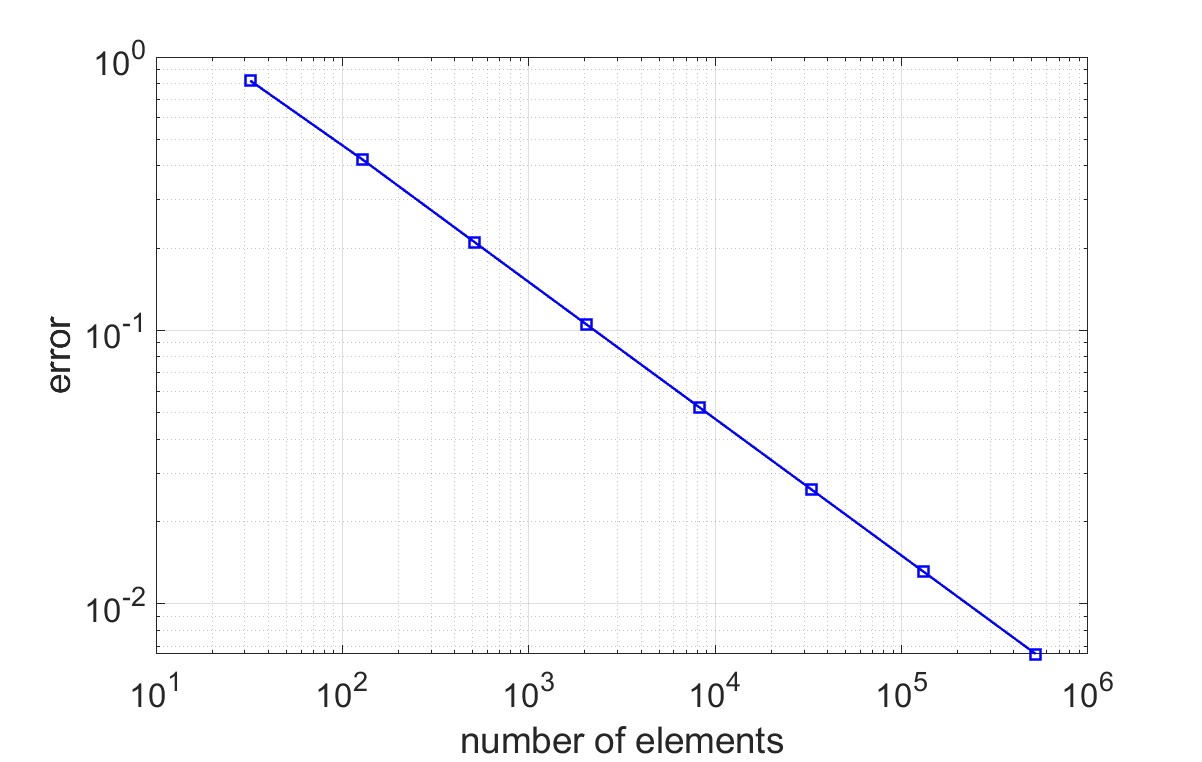}}
 \caption{Experiment~\ref{exp:mesh}. Convergence with respect to the number of elements in the mesh. Left for our model problem (MEQ.~1) and right for (MEQ.~2), respectively.}\label{fig:mesh}
\end{figure}
\end{experiment}

\section{Conclusions}
In this work, we studied some convergence properties of a discrete, relaxed $p(x)$-Poisson equation. First of all, we devised an iteration scheme that generates a sequence converging to a solution of the relaxed problem --- in the discrete as well as in the continuous case. Subsequently, we showed that on finite dimensional subspaces the solution of the relaxed problem converges to the solution of the original, unrelaxed equation. Finally, under suitable assumptions on the sequence of discrete spaces, we derived the convergence of the discrete solution to the continuous one. Admittedly, we have not constructed a computable sequence that converges to the solution of the continuous, unrelaxed $p(x)$-Poisson equation. However, this will be subject to a future research work. In particular, we want to design an algorithm that employs an adaptive inteplay of the damped Ka\v{c}anov iteration scheme, an enlargement of the relaxation parameter, and an hierarchical enrichment of the discrete spaces, and which generates a computable sequence with guaranteed convergence to the sought solution.    

\bibliographystyle{amsplain}
\bibliography{references}
\end{document}